\documentclass[fleqn,11pt, reqno]{amsart}
\usepackage{indentfirst, amsfonts, amsmath, amsthm, amssymb, amscd}
\usepackage{amsmath,amsfonts,amscd,bezier}
\usepackage[all]{xy}
\usepackage[utf8]{inputenc}
\usepackage{psfrag}
\usepackage{epstopdf}
\usepackage[dvips]{graphicx}
\usepackage{tikz} 

\usepackage{amsmath,amsfonts,latexsym}
\usepackage{amscd,amssymb,amsmath}
\usepackage{amsfonts}
\usepackage{amsbsy}
\usepackage{epsfig,afterpage}
\usepackage{psfrag}
\usepackage{graphicx}
\usepackage{colortbl}
\usepackage{cancel}
\usepackage[all]{xy}

\newtheorem{lemma}{Lemma}[section]
\newtheorem{theorem}[lemma]{Theorem}
\newtheorem{proposition}[lemma]{Proposition}

\newtheorem{corollary}[lemma]{Corollary}
\newtheorem{definition}[lemma]{Definition}

\newcommand{\bea}{\begin{eqnarray*}}
\newcommand{\eea}{\end{eqnarray*}}
\newcommand{\zz}[1]{}

\begin{document}

\thanks{$^{2}$Supported partially by FAPESP of Brazil, Grant 2011/22285-1, University of São Paulo, Brazil.}

\title{An Exact Sequence for Generalized String Links Over Surfaces}
\author[Juliana Roberta Theodoro de Lima]{Juliana Roberta Theodoro de Lima $^{1}$ }
\address{Mathematics Institute, Federal University of Alagoas,\\ Campus A. C. Simões, Maceio, Lourival Melo Motta Avenue, no number\\ CEP: 57072-970}
\email[Juliana Roberta Theodoro de Lima]{juliana.lima@im.ufal.br} 


\maketitle

\begin{abstract}
    In this work we extend Goldberg result \cite{Goldberg} for generalized string links over closed, connected and orientable surfaces of genus $g \geq 1$, i.e., different from the sphere (up to link-homotopy).
\end{abstract}

\section{Introduction}

In $1947$, E. Artin introduced the study of braids with his pioneering paper called \emph{Theory of Braids} (see \cite{Artin}), which is directly related to knots and links theory. Although braids, links and knots had already been discussed earlier, Artin showed two important results for the theory: the presentation and representation theorems for the braid group on the disk, namely $B_{n}$, also known as the \emph{Artin Braid Group}. For our purposes here, we focus on the first result: a presentation of a group is a way to represent a group by generators and relations. The braid group is a group of equivalence classes, where the equivalence relation is isotopy (or, more formally, ambient isotopy). However, in the same paper Artin proposed the idea of homotopy braids: essentially, it is the same set divided into equivalence classes using the equivalence relation of homotopy. The operation (concatenation) remains the same among braids. Accordingly, he posed the following questions: would the homotopy braids on the disk have the same properties, group structure and presentation as braid groups? Otherwise, what are its differences?

Goldsmith \cite{Goldsmith} answered all these questions: in fact, she proved that the group structures are different, making it explicit when certain types of braids are not trivial up to isotopy but trivial up to homotopy, called \emph{Link Homotopically Trivial b
Braids}, namely, $H_{n}$. Furthermore, she provided a explicit description of $H_{n}$ and a presentation for homotopy braid groups on the disk, denoted by $\widehat{B}_{n}$, which is the quotient of $B_{n}$ by $H_{n}$. Homotopy has been discussed since the beginning of the formalization of the studies of braid groups presented by Artin. However, homotopy braid theory was formalized by Milnor some years after Artin's seminal paper in \cite{Milnor} and it has been extended with the works \cite{Habegger, Levine, Yurasovskaya}. Moreover, there is still a slight difference between the concepts of string links given in \cite{Milnor, Yurasovskaya} and of homotopy braids given in \cite{Goldsmith}: string links are pure braids, either on the disk or on surfaces, with the monotonicity requirement relaxed, whereas homotopy braids are braids on the disk (not necessarily pure) with the monotonicity requirement relaxed. Consequently, we see that the most recent works are restricted to the pure case and, therefore, was reasonable to inquire about the general case, as provided in \cite{Lima2}, namely, $\widehat{B}_{n}(M)$, the group of generalized string links over surfaces. Still in \cite{Lima2}, was provided a exact sequence relating the braid group, $B_{n}(M)$, and the generalized string links over surfaces, $\widehat{B}_{n}(M)$, giving us a important tool for the famous open question about the left orderability of $B_{n}(M)$ \cite{Rolfsen, Gonzales2}. 

In \cite{Lima2} we formalize the definitons: generalized string links are homotopy braids with non trivial permutation induced by its strands. For the particular case when the permutation induced by the strands of a generalized string link is trivial, we call simply string links \cite{Yurasovskaya, Lima2}. 

When we put exact sequences and presentations (finitely) of groups together, it is possible to obtain precious results in orderability theory for braids as we see in \cite{Rolfsen, Yurasovskaya, Lima, Lima2}. From this notion and considering the fact generalized string links is a recently object to be explored (although it is already a structure used in other areas beyond mathematics\cite{fisica}), we present a new exact sequence, relating the normal subgroup of $\widehat{B}_{n}(M)$, called the group of string links over surfaces (pure), namely, $\widehat{PB}_{n}(M)$, $\widehat{PB}_{n}$ and the direct product of fundamental group of $M$, namely, $\pi_{1}(M)^{n}$, extending the exact sequence given in \cite{Goldberg}.

This paper is organized as follows: in Section \ref{Section2} we state results about generalized string links over surfaces that will be useful for our purposes. In Section \ref{Section3} we provide the main result of this paper, the exact sequence and the proof of its well definition, extending \cite{Goldberg}.

\section{Homotopy Generalized String Links \label{Section2}}

\subsection{Braids over Surfaces}

\begin{definition}\label{definicaoproblema}\rm{\cite[p.431]{Gonzales}}
Let $M$ be a closed surface, not necessarily orientable, and let $\mathcal{P}= \{P_{1} , \ldots ,P_{n}\}$ be a set of $n$ distinct points of $M$. A geometric braid over $M$ based at $\mathcal{P}$ is an $n$-tuple $\gamma= ( \gamma_{1},\ldots,\gamma_{n})$ of paths, $\gamma_{i}: [0,1] \rightarrow M$, such that:
\begin{itemize}
\item[(1)] $\gamma_{i}(0) = P_{i}, \hbox{for all} \ i= 1,\ldots,n$,
\item[(2)] $\gamma_{i}(1) \in \mathcal{P}, \ \hbox{for all} \ i = 1,\ldots,n$,
\item[(3)] $\{\gamma_{1}(t),\ldots,\gamma_{n}(t)\} \;\hbox{are $n$ distinct points of $M$, for all} \ t \in [0,1]$.
\end{itemize}
For all $i= 1,\ldots, n$, we will call $\gamma_{i}$ the $i$-th strands (or strings) of $\gamma$.
\end{definition}

\begin{figure}[h]
\center
\includegraphics[scale=0.5]{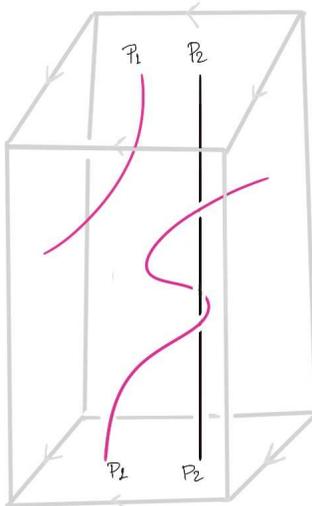}
\vspace{-1cm}
\caption{A braid "through the wall" $\beta$ over the $2$-dimensional torus. \label{lajm1}}
\end{figure}

We say two geometric braids $\beta$ and $\alpha$ are {\it{isotopic}} if there exists an ambient isotopy which deforms one to the other, with endpoints fixed during the deformation process. The set of all equivalence classes of geometric braids on $n$-strands on the surface $M$ forms a group called the \emph{braid group on $n$ strings on a surface $M$}, namely $B_{n}(M)$, equipped with the operation (product) called concatenation. The inverse of each braid $\gamma$ is given by the mirror reflection of $\gamma$. If the surface is the disk $\mathbb{D}$, then $B_{n}(\mathbb{D})$ is the Artin braid group $B_{n}$. 

\begin{figure}[h]
\center
\includegraphics[scale=0.4]{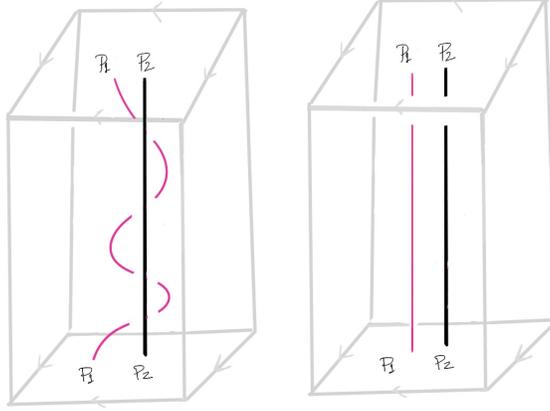}
\vspace{-0.5cm}
\caption{Two isotopic braids. \label{lajm5}}
\end{figure}

\begin{theorem}
{\rm\cite[Theorem 4.2]{Gonzales}}\label{Teorema1.1.3}
If $M$ is a closed, connected and orientable surface of genus $g \geq 1$ (different from the sphere), then $PB_{n}(M)$ admits the following presentation:
\begin{itemize}
\item[] {\rm\textbf{Generators}:} $\{a_{i, r}; \ 1\leq i \leq n, \  1\leq r \leq 2g\}\cup \{T_{j, k}; \  1\leq j < k \leq n\}$.
\item[] {\rm\textbf{Relations}:}
\end{itemize}
{\small
\noindent {\rm(PR1)}
$a_{n, 1}^{-1}a_{n, 2}^{-1}\cdots a_{n, 2g}^{-1}a_{n, 1}a_{n, 2}\cdots a_{n, 2g}= \displaystyle\prod_{i= 1}^{n-1}{T_{i, n-1}^{-1}T_{i, n}}$;

\noindent {\rm(PR2)}
$a_{i, r}A_{j, s}= A_{j, s}a_{i, r}, \hspace{5.6cm} 1\leq i<j \leq n; 1 \leq r, s \leq 2g; r\neq s$;

\noindent {\rm(PR3)}
$(a_{i, 1}\cdots a_{i, r})A_{j, r}(a_{i, r}^{-1}\cdots a_{i, 1}^{-1})A_{j, r}^{-1}= T_{i, j}T_{i, j-1}^{-1}$,
$\hspace{2.4cm} 1\leq i<j \leq n; 1\leq r \leq 2g$;

\noindent {\rm(PR4)}
$T_{i, j}T_{k, l}= T_{k, l}T_{i, j}, \hspace{3.4cm} 1\leq i<j<k<l \leq n \  \text{or} \ 1\leq i<k<l\leq j\leq n$;

\noindent {\rm(PR5)}
$T_{k, l}T_{i, j}T_{k, l}^{-1}= T_{i, k-1}T_{i, k}^{-1}T_{i, j}T_{i, l}^{-1}T_{i, k}T_{i, k-1}^{-1}T_{i, l}$,
$\hspace{3.1cm} 1\leq i<k \leq j<l \leq n$;

\noindent {\rm(PR6)}
$a_{i, r}T_{j, k}= T_{j, k}a_{i, r}, \hspace{2.8cm} 1\leq i<j<k \leq n \ \text{or} \ 1\leq j<k<i \leq n; 1\leq r \leq 2g $;

\noindent {\rm(PR7)}
$a_{i, r}(a_{j, 2g}^{-1}\cdots a_{j, 1}^{-1}T_{j, k}a_{j, 2g}\cdots a_{j, 1})= (a_{j, 2g}^{-1}\cdots a_{j, 1}^{-1}T_{j, k}a_{j, 2g}\cdots a_{j, 1})a_{i, r}$,
\hspace{0.05cm} $1\leq j<i\leq k \leq n$;

\noindent {\rm(PR8)}
$T_{j, n}= \left(\displaystyle\prod_{i= 1}^{j-1}{a_{i, 2g}^{-1}\cdots a_{i, 1}^{-1}T_{i, j-1}T_{i, j}^{-1}a_{i, 1}\cdots a_{i, 2g}}\right)a_{j, 1}\cdots a_{j, 2g}a_{j, 1}^{-1}\cdots a_{j, 2g}^{-1}$.}

\noindent Where 
\begin{eqnarray*}
A_{j, s}&=& a_{j, 1}\cdots a_{j, s-1}a_{j, s+1}^{-1}\cdots a_{j, 2g}^{-1}\\
T_{i,j}&=& \sigma_{i}\sigma_{i+1}\cdots\sigma_{j-2}\sigma_{j-1}^{2}\sigma_{j-2}\cdots\sigma_{i+1}\sigma_{i}
\end{eqnarray*}
\end{theorem}

\vspace{-0.2cm}

\begin{figure}[h]
\center
\includegraphics[scale=0.4]{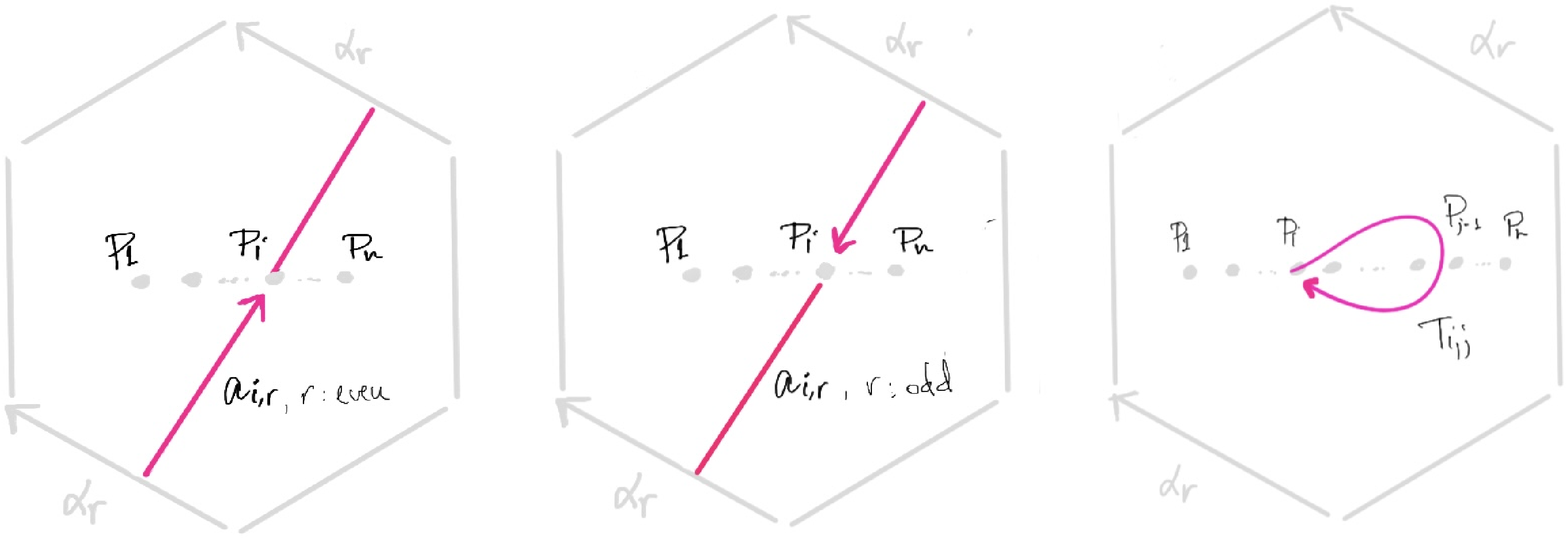}
\vspace{-0.6cm}
\caption{Generators of $PB_{n}(M)$. \label{lajm6}}
\end{figure}

\vspace{5cm}

\begin{theorem}
{\rm{\cite[Theorem 2.1]{Gonzales}}}\label{Teorema1.1.2}
If $M$ is a closed, orientable surface of genus $g \geq 1$, then $B_{n}(M)$ admits the following presentation:
\begin{itemize}
\item[] {\rm\textbf{Generators}:} $\sigma_{1},\ldots, \sigma_{n-1}, \mathit{a}_{1,1},\ldots,\mathit{a}_{1,2g}$.
\item[] {\rm\textbf{Relations}:}
\begin{itemize}
\item[{\rm(R1)}] $\sigma_{i}\sigma_{j}= \sigma_{j}\sigma_{i}, \hspace{10cm} |i-j|\geq 2$;
\item[{\rm(R2)}] $\sigma_{i}\sigma_{i+1}\sigma_{i} = \sigma_{i+1}\sigma_{i}\sigma_{i+1}, \hspace{7.6cm} 1\leq i \leq n-2$;
\item[{\rm(R3)}]
$\mathit{a}_{1,1}\cdots\mathit{a}_{1,2g}\mathit{a}_{1,1}^{-1}\cdots\mathit{a}_{1,2g}^{-1}=
\sigma_{1}\cdots\sigma_{n-2}\sigma_{n-1}^{2}\sigma_{n-2}\cdots\sigma_{1}$,
\item[{\rm(R4)}] $\mathit{a}_{1,r}A_{2, s}= A_{2, s}\mathit{a}_{1,r}, \hspace{7cm} 1\leq r, s \leq 2g;  \ r\neq s$;
\item[{\rm(R5)}] $(\mathit{a}_{1,1}\cdots\mathit{a}_{1,r})A_{2, r}=
\sigma_{1}^{2}A_{2, r}(\mathit{a}_{1,1}\cdots\mathit{a}_{1,r}), \hspace{5cm} 1\leq r \leq 2g$;
\item[{\rm(R6)}] $\mathit{a}_{1,r}\sigma_{i}= \sigma_{i}\mathit{a}_{1,r}, \hspace{8.1cm} 1\leq r \leq 2g; \  i\geq 2$.
\end{itemize}
\end{itemize}
\end{theorem}

If $M$ is the unit disk, the presentation of Theorem $\ref{Teorema1.1.2}$ is reduced to Artin´s presentations given in \cite{Artin}.

\begin{figure}[h]
\center
\includegraphics[scale=0.45]{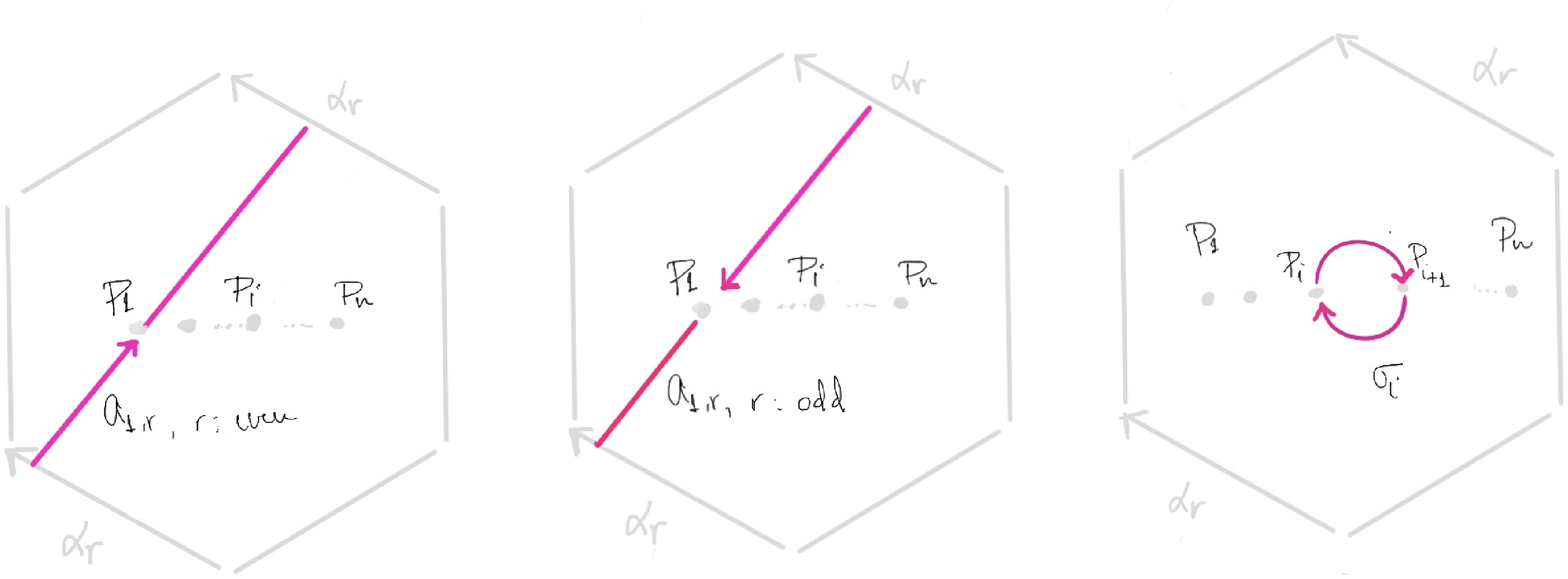}
\vspace{-0.5cm}
\caption{Generators of $B_{n}(M)$. \label{lajm10}}
\end{figure}

\subsection{Generalized String Links over Surfaces}

\begin{definition}
\cite[Definition $3.1$]{Lima2}\label{def3.1}
A \emph{generalized string link $\sigma$ on $n$ strands on a surface $M$} is a smooth or piecewise linear proper embedding $\sigma: \coprod_{i= 1}^{n}{I_{i}} \rightarrow M \times I,$ which fulfills the two following conditions:
\begin{itemize}
\item[{\rm(i)}] $\sigma|_{(I_{i}(0))}= (P_{i},0)$,
\item[{\rm(ii)}] $\sigma|_{(I_{i}(1))} \in \{(P_{1},1),\ldots,(P_{n},1)\}$,
\end{itemize}
where $I_{i}(t)=t\in I_{i}$, for all $t$ and for all $i=1,\ldots,n$.
\end{definition}

\begin{figure}[h!]
\center
\includegraphics[scale=0.35]{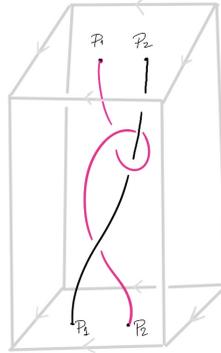}
\vspace{-0.5cm}
\caption{A generalized string link $\sigma$ over the $2$-dimensional torus. \label{lajm2}}
\end{figure}

Here, we orient the strands downwards from $M \times \{0\}$ to $M \times\{1\}$. Besides, an \emph{ambient isotopy} between generalized string links $\sigma$ and $\sigma^{\prime}$ is an orientation-preserving diffeomorphism of $M \times I$ which maps $\sigma$ onto $\sigma^{\prime}$, keeping the boundary $M \times \{0,1\}$ point-wise fixed and is isotopic to the identity, relative to $M \times \{0,1\}$.
When $\sigma|_{(I_{i}(0))}= (P_{i},1)$, we just obtain a string link, i.e., a string link is a pure generalized string link described in \cite{Yurasovskaya}. When the surface in question is the disk $\mathbb{D}$, we have the homotopy braids described in \cite{Goldsmith}.

\begin{definition}\label{linkhomotopia1}\cite{Habegger} 
Two generalized string links $\sigma$ and $\sigma^{\prime}$ are \emph{link-homotopic} if there is a homotopy of the strings in $M \times I$, fixing $M \times \{0,1\}$ and deforming $\sigma$ to $\sigma^{'}$, such that the images of different strings remain disjoint during the deformation. During the course of deformation, each individual strand is allowed to pass through itself but not through other strands. 
\end{definition}

Equivalently, 

\begin{definition}\label{linkhomotopia2}\cite{Habegger,Milnor,Levine}
We say \emph{link-homotopy} is an equivalence relation on generalized string links that is generated by a finite sequence of ambient isotopies of $M \times I$ fixing $M \times \{0,1\}$, and local \emph{crossing changes} of arcs from the \emph{same strand} of a generalized string link called {\it{link-homotopy moves}}. 
\end{definition}

\begin{figure}[h]
\center
\includegraphics[scale=0.45]{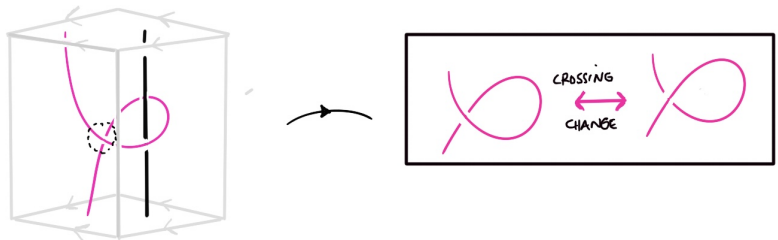}
\vspace{-0.3cm}
\caption{A crossing change. \label{lajm3}}
\end{figure}

\vspace{5cm}

The property of the local crossing changes consists in considering the undercrossing and the overcrossing as the same crossing (in the same strand), as shown in Figure  \ref{lajm3}, i.e, a crossing change for generalized string links remains the same as defined for string links.

\begin{figure}[h]
\center
\includegraphics[scale=0.4]{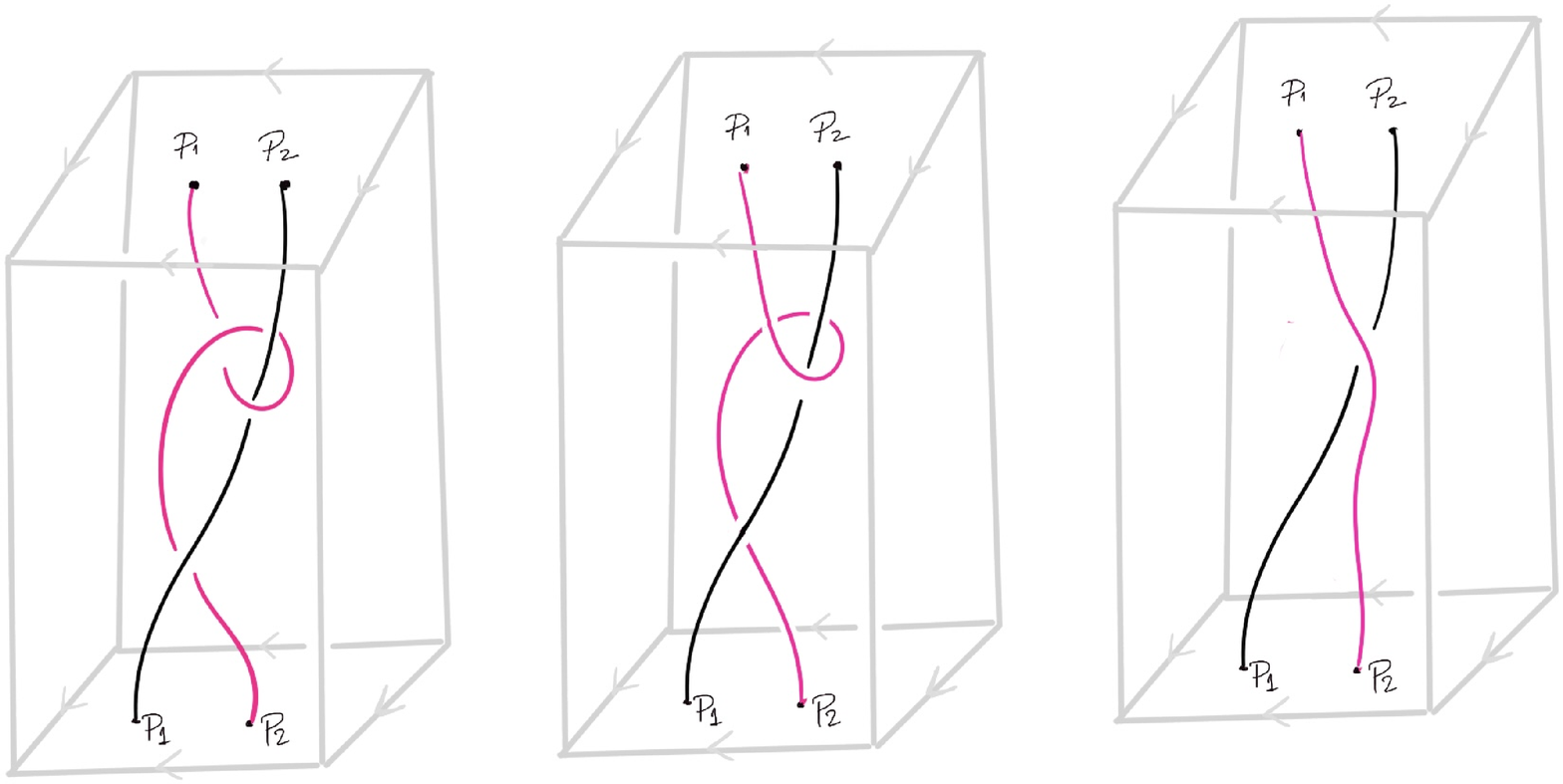}
\caption{A generalized string link $\sigma$ link homotopic to a braid up to a crossing change. \label{lajm4}}
\end{figure}

Define $H_{n}(M)$ the set of all pure braids in $PB_{n}(M)$ which are link-homotopic to the trivial braid. This set is called the set of the \emph{link-homotopically trivial braids}. In symbols:
$$H_{n}(M)= \{\beta \in PB_{n}(M); \;\;\beta \sim 1\},$$ where $\sim$ denotes the link-homotopy equivalence relation.\\

\begin{proposition}{\rm\cite[p.6]{Lima2}}\label{Corolario2.1}
$H_{n}(M)$ is the smallest normal subgroup of $PB_{n}(M)$ generated by $(LH1)$. In symbols:
$$H_{n}(M)= \langle \{[t_{i,j},t_{i,j}^{h}], \ 1 \leq i < j \leq n, \ h \in \mathbb{F}(2g+n-i)\} \rangle^{N},$$
where $\langle \quad \rangle^{N}$ denotes the normal closure.\\
\end{proposition}

\begin{proposition}{\rm\cite[Proposition 3.9]{Yurasovskaya}} \label{prop2.1}
Under concatenation $\widehat{PB}_{n}(M)$ is a group isomorphic to the quotient of the pure braid group $PB_{n}(M)$ by the subgroup of link-homotopically trivial braids $H_{n}(M)$: $$\widehat{PB}_{n}(M)= \frac{PB_{n}(M)}{H_{n}(M)}.$$
\end{proposition}

\begin{theorem}\rm{\cite[Theorem 6.3]{Yurasovskaya}}\label{Teorema2.1.5}
Let $M$ be a closed, compact, connected and orientable surface of genus $g \geq 1$. The group of homotopy string links $\widehat{PB}_{n}(M)$ admits the presentation:
\begin{itemize}
\item[]{\bf{Generators:}} $\{a_{i,r}; \;\;1\leq i \leq n; \;\; 1 \leq r \leq 2g\}\cup \{t_{j,k};\;\; 1 \leq j < k \leq n\}.$
\item[]{\bf{Relations:}}

{\small
\noindent {\rm(LH1)} $[t_{i,j},t_{i,j}^{h}]= 1$, \hspace{8.4cm} $h \in \mathbb{F}(2g + n- i)$;

\noindent{\rm(PR1)} $a_{n, 1}^{-1}a_{n, 2}^{-1}\cdots a_{n, 2g}^{-1}a_{n, 1}a_{n, 2}\cdots a_{n, 2g}=
 \displaystyle\prod_{i= 1}^{n-1}{T_{i, n-1}^{-1}T_{i, n}}$;
 
\noindent{\rm(PR2)} $a_{i, r}A_{j, s}= A_{j, s}a_{i, r}, \hspace{2cm} 1\leq i<j \leq n, \; 1 \leq r \leq 2g;\; 1 \leq s \leq 2g-1; \;\; r\neq s$;

\noindent{\rm(PR3)} $(a_{i, 1}\cdots a_{i, r})A_{j, r}(a_{i, r}^{-1}\cdots a_{i, 1}^{-1})A_{j, r}^{-1}= T_{i, j}T_{i, j-1}^{-1}$, $\hspace{0.6cm} 1\leq i<j \leq n, \;\; 1\leq r \leq 2g-1$;

\noindent{\rm(PR4)} $T_{i, j}T_{k, l}= T_{k, l}T_{i, j}, \hspace{2.4cm} 1\leq i<j<k<l \leq n \;\; \text{or} \;\; 1\leq i<k<l\leq j\leq n$;

\noindent{\rm(PR5)} $T_{k, l}T_{i, j}T_{k, l}^{-1}= T_{i, k-1}T_{i, k}^{-1}T_{i, j}T_{i, l}^{-1}T_{i, k}T_{i, k-1}^{-1}T_{i, l}$,
$\hspace{2.2cm} 1\leq i<k \leq j<l \leq n$;

\noindent{\rm(PR6)} $a_{i, r}T_{j, k}= T_{j, k}a_{i, r}, \hspace{1.5cm} 1\leq i<j<k \leq n \;\; \text{or} \;\; 1\leq j<k<i \leq n, \;\; 1\leq r \leq 2g $;

\noindent{\rm(PR7)} $a_{i, r}(a_{j, 2g}^{-1}\cdots a_{j, 1}^{-1}T_{j, k}a_{j, 2g}\cdots a_{j, 1}) = (a_{j, 2g}^{-1}\cdots a_{j, 1}^{-1}T_{j, k}a_{j, 2g}\cdots a_{j, 1})a_{i, r}$, $1\leq j<i \leq k \\ \leq n$;

\noindent {\rm(PR8)} $T_{j,n}= \left(\displaystyle\prod_{i= 1}^{j-1}{a_{i, 2g}^{-1}\cdots a_{i, 1}^{-1}T_{i, j-1}T_{i, j}^{-1}a_{i, 1}\cdots a_{i, 2g}}\right)a_{j, 1}\cdots a_{j, 2g}a_{j, 1}^{-1}\cdots a_{j, 2g}^{-1}$;}
\end{itemize}
 where 
 \begin{eqnarray*}
 A_{j, s}&=& a_{j, 1}\cdots a_{j, s-1}a_{j, s+1}^{-1}\cdots a_{j, 2g}^{-1}\\
 T_{i,j}&=& t_{i,j}\cdots t_{i,i+1}.
 \end{eqnarray*}
\end{theorem}

\begin{figure}[h]
\center
\includegraphics[scale=0.4]{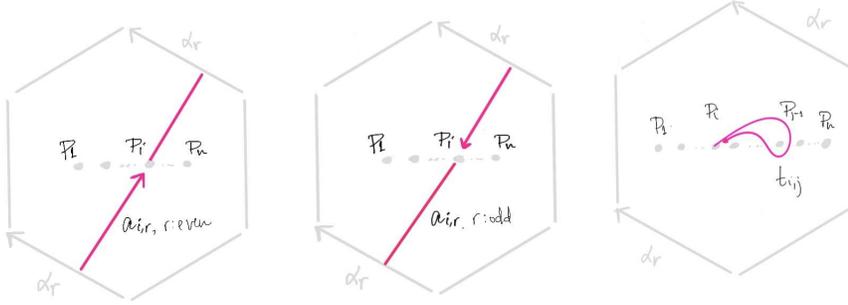}
\vspace{-0.2cm}
\caption{Generators of $\widehat{PB}_{n}(M)$.}
\end{figure}

\begin{proposition}{\rm\cite[Theorem 3.7]{Lima2}} \label{teo3.7}
Under concatenation $\widehat{B}_{n}(M)$ is a group isomorphic to the quotient of the pure braid group $PB_{n}(M)$ by the subgroup of link-homotopically trivial braids $H_{n}(M)$: $$\widehat{B}_{n}(M)= \frac{B_{n}(M)}{H_{n}(M)}.$$
\end{proposition}

\begin{theorem}\label{Teorema4.4.1}\cite[Theorem 3.9]{Lima2}
Let $M$ be a closed, orientable surface of genus $g\geq 1$. The group of link-homotopy classes of generalized string links over $M$, namely $\widehat{B}_{n}(M)$, admits the following presentation:
\begin{itemize}
\item[] {\bf{Generators:}} $\{a_{1,1},\ldots,a_{1,2g}\} \cup \{\sigma_{1},\ldots,\sigma_{n-1}\}$;
\item[] {\bf{Relations:}}
\begin{itemize}
\item[{\rm(LH)}] $[{t_{1,j}},{t^{h}_{1,j}}]=1 \hspace{8.5cm} h \in \mathbb{F}(2g+n-1)$;
\item[{\rm(R1)}] $\sigma_{i}\sigma_{j}= \sigma_{j}\sigma_{i}  \hspace{9.9cm} |i-j|\geq 2$;
\item[{\rm(R2)}] $\sigma_{i}\sigma_{i+1}\sigma_{i} = \sigma_{i+1}\sigma_{i}\sigma_{i+1}  \hspace{7.6cm} 1\leq i \leq n-2$;
\item[{\rm(R3)}]
$\mathit{a}_{1, 1}\cdots\mathit{a}_{1, 2g}\mathit{a}_{1, 1}^{-1}\cdots\mathit{a}_{1, 2g}^{-1}=
\sigma_{1}\cdots\sigma_{n-2}\sigma_{n-1}^{2}\sigma_{n-2}\cdots\sigma_{1}$
\item[{\rm(R4)}] $\mathit{a}_{1, r}A_{2, s}= A_{2, s}\mathit{a}_{1, r} \hspace{4.9cm} 1\leq r \leq 2g \ 1\leq s \leq 2g-1, \ r\neq s$;
\item[{\rm(R5)}] $(\mathit{a}_{1,1}\cdots\mathit{a}_{1, r})A_{2, r}=
\sigma_{1}^{2}A_{2, r}(\mathit{a}_{1, 1}\cdots\mathit{a}_{1, r}) \hspace{4.7cm} 1\leq r \leq 2g-1$;
\item[{\rm(R6)}] $\mathit{a}_{1, r}\sigma_{i}= \sigma_{i}\mathit{a}_{1, r} \hspace{8.3cm} 1\leq r \leq 2g; \  i\geq 2$;
\end{itemize}
\end{itemize}where:
\begin{eqnarray*} 
t_{1,j}&=& \sigma_{1}\cdots\sigma_{j-2}\sigma_{j-1}^{2}\sigma_{j-2}^{-1} \cdots \sigma_{1}^{-1}, \hspace{0.3cm} j=2,\ldots,n, \\ 
A_{2,s}&=&\sigma_{1}^{-1}(a_{1,1}\cdots a_{1,s-1}a_{1,s+1}^{-1}\cdots a_{1,2g}^{-1})\sigma_{1}^{-1},  \hspace{0.3cm} s=1,\ldots,2g-1.
\end{eqnarray*}
\end{theorem}

\begin{figure}[h]
\center
\includegraphics[scale=0.45]{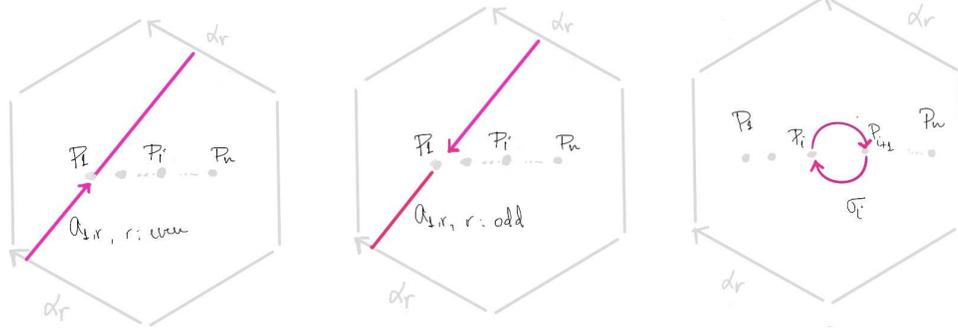}
\caption{Generators of $\widehat{B}_{n}(M)$.}
\end{figure}

\begin{corollary}\label{Prop4.2.7}\cite[p.18]{Lima2}
Under the conditions above, we have the well defined short exact sequence:
$$\xymatrix{
                      1 \ar[r]   &   {H}_{n}(M) \ar[r]^{i}   &   {B}_{n}(M) \ar[r]^{\hspace{0.2cm}\psi}   &   \widehat{B}_{n}(M) \ar[r]   &   1
},$$ where $i$ is the inclusion homomorphism.
\end{corollary} 

\vspace{0.2cm}

\begin{proposition}\label{Prop4.2.7}\cite[Proposition 3.8]{Lima2}
$\widehat{PB}_{n}(M)$ is a normal subgroup of $\widehat{B}_{n}(M)$. Moreover, under the homomorphism $\psi$ defined previously, we have the well defined short exact sequence:
$$\xymatrix{
                      1 \ar[r]   &   \widehat{PB}_{n}(M) \ar[r]^{i}   &   \widehat{B}_{n}(M) \ar[r]^{\hspace{0.2cm}\psi}   &   \Sigma_{n} \ar[r]   &   1
},$$ where $i$ is the inclusion homomorphism.
\end{proposition} 

From now on, let $M$ be a closed, connected and orientable surface of genus $g\geq 1$. Given a string link $\hat{\sigma}= [(\sigma_{1},\ldots,\sigma_{n})]$ over $M$ in $\widehat{PB}_{n}(M)$, we can consider, for all $i=1,\ldots,n$ the loop $\mu_{i}$ in $M$ constructed as follows: take the $i$-th string $\sigma_{i}$ (which is a path in $M \times [0,1]$) and call as $\mu_{i}$ its projection over the first coordinate (i.e., over $M$). Since $\hat{\sigma} \in \widehat{PB}_{n}(M)$, $\mu_{i}$ is a loop in $M$ based at $P_{i}$, for all $i=1,\ldots,n$ which represents an element of $\pi_{1}(M, P_{i})\simeq \pi_{1}(M)$. This defines an epimorphism $\hat{\theta}_{n}: \widehat{PB}_{n}(M)\rightarrow \pi_{1}(M)^{n}$ which sends $\hat{\sigma}= [(\sigma_{1},\ldots,\sigma_{n})]$ to $\mu= (\mu_{1},\ldots, \mu_{n})$.

\begin{lemma}\cite[Lemma 3.1]{Lima}\label{lema3.1.1}
$\hat{\theta}_{n}$ is a well defined surjective homomorphism.
\end{lemma}

\section{An Exact Sequence for Link-Homotopy Braid Groups \label{Section3}}

In this section we extend the result provided by Charles H. Goldberg in \cite[Theorem 1]{Goldberg} for generalizing string links over surfaces. To prove this result, we will use as a tool the surjective homomorphism $\hat{\theta}_{n}$ of Lemma \ref{lema3.1.1}.

\subsection{Main Theorem}

Let us consider a closed, connected and orientable surface $M$, of genus $g\geq1$, i.e., $M$ is different of the sphere. Let $\mathcal{P}= \{P_{1}, P_{2},\ldots, P_{n}\}$ be a set of $n$ distinct fixed points chosen arbitrarily in the interior of $M$.Now, let us define the following map: $$\hat{f_{n}}: \widehat{PB}_{n}(\mathbb{D}) \rightarrow \widehat{PB}_{n}(M)$$ given by $\hat{f_{n}}(\hat{\beta})= \hat{\beta}$, for each $\hat{\beta}$ in $\widehat{PB}_{n}(\mathbb{D})$, where $\mathbb{D}$ denotes the unit disk. Also, we have $\hat{\theta}_{n}: \widehat{PB}_{n}(M) \rightarrow \pi_{1}(M)^{n}$ given by $\hat{\theta}_{n}(\hat{\alpha})= \hat{\theta}_{n}([(\alpha_{1},\ldots,\alpha_{n})])= (\mu_{1},\ldots,\mu_{n})$, where each $\mu_{i}$ is the strand $\alpha_{i}$ of $\hat{\alpha}$ viewed as a loop in the fundamental group of $M$, $i=1,\ldots,n$, for all $\hat{\alpha}=[(\alpha_{1},\ldots,\alpha_{n})]$ in $\widehat{PB}_{n}(M)$. By Lemma \ref{lema3.1.1}, $\hat{\theta}_{n}$ is a well defined surjective homomorphism.

\begin{theorem}\label{Teorema3.3.1}
If $M$ is a closed, connected and orientable surface, of genus $g\geq 1$, then in the following sequence of groups (not necessarily abelian):
$$\xymatrix{
                      1 \ar[r]   &   \widehat{PB}_{n}(\mathbb{D}) \ar[r]^{\hspace{-0.3cm}\hat{f_{n}}}   &   \widehat{PB}_{n}(M) \ar[r]^{\hspace{0.2cm}\hat{\theta}_{n}}   &   \pi_{1}(M)^{n} \ar[r]   &   1
},$$ the kernel of each homomorphism is equal to the normal closure of the image of the previous homomorphism in the sequence, i.e., $\ker(\hat{\theta}_{n})= \langle Im\hat{f_{n}}\rangle^{N}$.
\end{theorem}

\subsection{The well definition of $\hat{f}_{n}$ and proof of Theorem \ref{Teorema3.3.1}}

\begin{lemma}\label{lema3.3.2}
The map $\hat{f}{_{n}}$ is a well defined injective homomorphism.
\end{lemma}
\begin{proof} Recall the inclusion homomorphism $f_{n}: PB_{n}(\mathbb{D}) \rightarrow PB_{n}(M)$ defined by Birman in \cite{Birman2}.
Now, let $\beta$ be an element link-homotopically trivial in $PB_{n}(\mathbb{D})$. Clearly, $f_{n}(\beta)= \beta$ is a link-homotopically trivial element of $PB_{n}(M)$, i.e., $f_{n}(H_{n}(\mathbb{D})) \subseteq H_{n}(M)$. Thus, $\hat{f}_{n}$ is the well defined homomorphism induced by the injection $f_{n}$. Now, we need to prove that $\hat{f}_{n}$ is injective. Indeed, let $\hat{\beta} \in \widehat{PB}_{n}(\mathbb{D})$ such that $\hat{f}_{n}(\hat{\beta})= 1$, where $1$ denotes the identity in $\widehat{PB}_{n}(M)$, i.e., $1= H_{n}(M)$. So we have $\hat{\beta} \in \widehat{PB}_{n}(\mathbb{D})$ and $\hat{\beta} \in H_{n}(M)$. By Goldsmith in \cite{Goldsmith}, we have $\hat{\beta} \in H_{n}(\mathbb{D})$. Therefore, $\ker{\hat{f}_{n}}= H_{n}(\mathbb{D})$, i.e., $\hat{f}_{n}$ is injective.
\end{proof}

Let us consider the following diagram:
\begin{eqnarray}\label{diaiaiaiaiaiaaiia}\xymatrix{
    1 \ar[r] &  PB_{n}(\mathbb{D}) \ar[d]^{p_{1}} \ar[r]^{f_{n}}               & PB_{n}(M)       \ar[d]^{p_{2}} \ar[r]^{\theta_{n}} & \pi_{1}(M)^{n}   \ar[d]^{id}  \ar[r]    &  1\\
    1 \ar[r] & \widehat{PB}_{n}(\mathbb{D})           \ar[r]^{\hat{f}_{n}}         & \widehat{PB}_{n}(M) \ar[r]^{\hat{\theta}_{n}}          & \pi_{1}(M)^{n}           \ar[r]    &  1 }
\end{eqnarray} where $p_{1},p_{2}$ are the respective projections and $id$ is the identity in $\pi_{1}(M)$.

We claim that this diagram is commutative. Indeed, $\forall \beta \in PB_{n}(\mathbb{D})$,
$$\hat{f}_{n}\circ p_{1}(\beta)= \hat{f}_{n}([\beta])= [\beta]\;\; \mbox{\rm and} \;\; p_{2}\circ f_{n}(\beta)= p_{2}(\beta)= [\beta],$$
that is, $\hat{f}_{n}\circ p_{1}= p_{2}\circ f_{n}$. Moreover, $\forall \alpha \in PB_{n}(M)$, 
$$\hat{\theta}_{n}\circ p_{2}(\alpha)= \hat{\theta}_{n}([\alpha])= (\mu_{1},\ldots,\mu_{n})\;\; \mbox{\rm and} \;\;id\circ \theta_{n}(\alpha)= id(\mu_{1},\ldots,\mu_{n})= (\mu_{1},\ldots,\mu_{n}),$$ i.e., $\hat{\theta}_{n}\circ p_{2}=id\circ \theta_{n}$.

\begin{lemma}\label{lema3.3.3}
${\rm{Im}}(\hat{f}_{n})\subseteq \ker(\hat{\theta}_{n})$.
\end{lemma}
\begin{proof} Let $\hat{\alpha} \in {\rm{Im}}(\hat{f}_{n})$. So, there is an element $\hat{\beta}\in \widehat{PB}_{n}(\mathbb{D})$ such that $\hat{\alpha}= \hat{f}_{n}(\hat\beta)$. 
Thus, $\hat{\theta}_{n}\circ\hat{f}_{n}(\hat{\beta})= \hat{\theta}_{n}(\hat{\alpha})$. Since $\hat\beta$ is in
$\widehat{PB}_{n}(\mathbb{D})$ and $p_{1}$ is surjective, there is $\beta$ in $PB_{n}(\mathbb{D})$ such that $p_{1}(\beta)= \hat{\beta}$. So, $\hat{\theta}_{{n}}\circ \hat{f}_{n} \circ p_{1}(\beta)= \hat{\theta}_{n}(\hat{\alpha})$, which implies, $\hat{\theta}_{n}\circ p_{2}\circ f_{n}(\beta)= \hat{\theta}_{n}(\hat{\alpha})$, i.e., $\theta_{n}\circ f_{n}(\beta)= \hat{\theta}_{n}(\hat{\alpha})$. By \cite{Goldberg}, we have ${\rm{Im}}(f_{n})\subseteq \ker(\theta_{n})$, i.e., $\hat{\theta}_{n}(\hat\alpha)=1$ and then, $\hat\alpha \in \ker(\hat{\theta}_{n})$.
\end{proof}

\begin{lemma}\label{lema3.3.4}
$\ker(\hat{\theta}_{n})\subseteq \langle {\rm{Im}}(\hat{f}_{n})\rangle^{N}$.
\end{lemma}
\begin{proof} Let $\hat{\gamma}$ be an element in $\ker(\hat{\theta}_{n})$. Thus, $\hat{\gamma} \in \widehat{PB}_{n}(M)$ and $\hat{\theta}_{n}(\hat\gamma)=1$. By Theorem \ref{Teorema2.1.2}, each string link is link homotopic to a pure braid. Let $\gamma \in PB_{n}(M)$ be such pure braid. So, ${\theta}_{n}(\gamma)=1$, where $1$ denotes the identity in $\pi_{1}(M)^{n}$. By \cite[Theorem 1]{Goldberg}, we have $\gamma= \prod_{k}{\alpha_{k}\beta_{k}\alpha_{k}^{-1}}$, with $\alpha_{k} \in PB_{n}(M), \beta_{k} \in {\rm{Im}}(f_{n})$, i.e., $\beta_{k}= f_{n}(\gamma_{k})$, for some $\gamma_{k} \in PB_{n}(\mathbb{D})$ with $p_{1}(\gamma_{k})= \hat{\gamma}_{k}$, and then, $\hat{f}_{n}\circ p_{1}(\gamma_{k})= \hat{f}_{n}(\hat{\gamma}_{k})$.
Since the diagram (\ref{diaiaiaiaiaiaaiia}) commutes, 
$$\hat{f}_{n}(\hat{\gamma}_{k})= \hat{f}_{n}\circ p_{1}(\gamma_{k})= p_{2}\circ f_{n}(\gamma_{k})= p_{2}(\beta_{k})= \hat{\beta}_{k},$$ i.e., $\hat{\beta}_{k} \in {\rm{Im}}(\hat{f}_{n}).$
Therefore, 
$$\hat{\gamma}= p_{2}(\gamma)= p_{2}\left(\prod_{k}{\alpha_{k}\beta_{k}\alpha_{k}^{-1}}\right)= \prod_{k}p_{2}(\alpha_{k})p_{2}(\beta_{k})p_{2}(\alpha_{k})^{-1}= \prod_{k}\hat{\alpha_{k}}\hat{\beta_{k}}\hat{\alpha_{k}}^{-1},$$ with $\hat{\alpha}_{k} \in \widehat{PB}_{n}(M), \hat{\beta}_{k} \in {\rm{Im}}(\hat{f}_{n})$. 
\end{proof}

\begin{proof}{\bf Proof of Theorem \ref{Teorema3.3.1}:} Follows from Lemmas \ref{lema3.3.2}, \ref{lema3.3.3}  and \ref{lema3.3.4}. 
\end{proof}

\section*{\bf{Acknowledgment:}} A special thanks to FAPESP- Fundação de Amparo a Pesquisa do Estado de São Paulo for the research financial support from the Phd´s process number $2011/22285-1$.

\small
\printindex

\end{document}